\documentclass[a4paper]{article}
\usepackage{amsthm,amssymb,amsmath,enumerate,graphicx,epsf}

\newcommand{\NOTESON}{0}
\newcommand{\Debug}{0}
\newcommand{\COLORON}{0}

\usepackage[usenames]{color} 

\newcommand{\comment}[1]{}
\newcommand{\COMMENT}[1]{}

\newcommand{\defi}[1]{{\color{Green}\emph{#1}}}

\newcommand{\acknowledgement}{\section*{Acknowledgement}}

\newtheorem{proposition}{Proposition}[section]
\newtheorem{definition}[proposition]{Definition}
\newtheorem{theorem}[proposition]{Theorem}
\newtheorem{corollary}[proposition]{Corollary}
\newtheorem{lemma}[proposition]{Lemma}

\newtheorem*{ufpConjecture}{{\color{red}Unfriendly Partition Conjecture}}

\newtheorem{examp}{Example}[section]

\newcommand{\FIG}{0}

\ifnum \NOTESON = 1 \newcommand{\note}[1]{ 

	\ 

	{\color{Blue} NOTE: \color{Turquoise}{\small  \tt \begin{minipage}[c]{0.8\textwidth}  #1 \end{minipage} \ignorespacesafterend }} 
	
	\ 
	
	}
\else \newcommand{\note}[1]{} \fi

\ifnum \Debug = 1 
\else  \fi

\ifnum \FIG = 1 
\else  \fi

\ifnum \Debug = 1 \usepackage[notref,notcite]{showkeys}
\fi

\ifnum \COLORON = 0 \renewcommand{\color}[1]{}
\fi

\newcommand{\cc}{\mathcal C}

\newcommand{\cf}{\mathcal F}

\newcommand{\cj}{\mathcal J}

\newcommand{\cu}{\mathcal U}

\newcommand{\cx}{\mathcal X}
\newcommand{\cy}{\mathcal Y}
\newcommand{\cz}{\mathcal Z}
\newcommand{\cw}{\mathcal W}

\newcommand{\oo}{\ensuremath{\omega}}

\newcommand{\alp}{\ensuremath{\alpha}}
\newcommand{\bet}{\ensuremath{\beta}}
\newcommand{\gam}{\ensuremath{\gamma}}
\newcommand{\del}{\ensuremath{\delta}}

\newcommand{\kap}{\ensuremath{\kappa}}
\newcommand{\lam}{\ensuremath{\lambda}}

\newcommand{\ccc}{\ensuremath{\mathcal C}}

\newcommand{\sm}{\backslash}
\newcommand{\restr}{\upharpoonright}

\newcommand{\sgl}[1]{\ensuremath{\{#1\}}}

\newcommand{\oseq}[2]{\ensuremath{(#1_\beta)_{\beta< #2}}} 
 
\newcommand{\ofml}[2]{\ensuremath{\{#1_\beta\}_{\beta< #2}}} 

\newcommand{\g}{\ensuremath{G\ }}
\newcommand{\G}{\ensuremath{G}}

\newcommand{\Lr}[1]{Lemma~\ref{#1}}
\newcommand{\Tr}[1]{Theorem~\ref{#1}}
\newcommand{\Sr}[1]{Section~\ref{#1}}
\newcommand{\Prr}[1]{Proposition~\ref{#1}}

\newcommand{\lfg}{locally finite graph}

\renewcommand{\iff}{if and only if}

\newcommand{\fe}{for every}

\newcommand{\st}{such that}

\newcommand{\ti}{there is}

\newcommand{\labtequ}[2]{ \begin{equation} \label{#1} 	\begin{minipage}[c]{0.9\textwidth}  #2 \end{minipage} \ignorespacesafterend \end{equation} } 

\newcommand{\mySection}[2]{}

\newcommand{\ufp}{unfriendly partition}
\newcommand{\smp}{strongly maximal partition}
\newcommand{\prfr}{pre-partitionable}

\newcommand{\opp}{opponent}
\newcommand{\opps}{opponents}
\newcommand{\klasse}{finitely closed}

\newcommand{\clx}[1]{\ensuremath{\overline {#1}}}
\newcommand{\clu}{\clx{\cu}}

\newcommand{\vinf}{\ensuremath{V_{\infty}}}

\newcommand{\rk}{r}

\title{Every rayless graph has an unfriendly partition}
\author{Henning Bruhn \qquad Reinhard Diestel\\[3pt] ${}^*$Agelos Georgakopoulos$\,$\thanks{Supported by a GIF grant.}\\[3pt] Philipp Spr\"ussel 
\bigskip \\
  {Mathematisches Seminar}\\
  {Universit\"at Hamburg}\\
  {\small Bundesstr.\ 55, 20146, Germany}
}

\date{\small 12 December, 2009}
\begin{document}
\maketitle

\begin{abstract}
We prove that every rayless graph has an unfriendly partition.
\end{abstract}

\section{Introduction}

A bipartition of the vertex set of a graph is \defi{unfriendly} if every vertex has at least as many 
neighbours in the other class as in its own. The following conjecture is one of the best-known open problems in infinite graph theory. 

\begin{ufpConjecture}\label{ufpc}
Every countable graph admits an unfriendly partition of its vertex set.
\end{ufpConjecture} 

Clearly, every finite graph 
has an unfriendly partition: just take any bipartition that maximizes the 
number of edges between the partition classes. For infinite graphs however,  few results are known. Shelah and Milner \cite{MilShel} constructed uncountable graphs that admit no unfriendly partitions, thereby disproving the original folklore conjecture that every graph has an \ufp. (They attribute this conjecture to R.~Cowan and W.~Emerson, unpublished.) It is an easy exercise in compactness to deduce from the finite theorem that every \lfg\ has an \ufp, see \cite{diestelBook05}. Aharoni, Milner and Prikry \cite{AhMiUnf} strengthened this fact by proving that every graph with only finitely many vertices of infinite degree has an \ufp.

At the other end of the spectrum, it is easy to prove that countable graphs in which only finitely many vertices have finite degree have \ufp s. Since the counterexamples from~\cite{MilShel} have no vertices of finite degree, the countability assumption here cannot be dropped.

The main aim of this paper is to prove that all rayless graphs, countable or not, have \ufp s:

\begin{theorem}\label{rayless}
Every rayless graph has an unfriendly partition.
\end{theorem} 

For the proof of this theorem we used a tool developed by Schmidt~\cite{schmidt83},
which assigns to every rayless graph an ordinal number,  its \defi{rank}.
This rank enables us to prove Theorem~\ref{rayless} by transfinite induction.

Closer inspection of the proof reveals that it does not
depend on the graphs being rayless, but only on the existence of such a kind of rank. Extending the method to more general rank functions, one can indeed strengthen Theorem~\ref{rayless} in various ways. One such strengthening obtained in Section~\ref{secLarge} is this:

\begin{theorem}\label{brush}
Every graph not containing a ray linked disjointly to infinitely many vertices of infinite degree has an \ufp.
\end{theorem}

\section{Definitions and statement of main result} \label{secDef}

Recall that
a \defi{cofinal} subset of an ordered set $A$ is a subset $B$  
such that for every $a\in A$ there is a $b \in B$ such that $a \leq b$. The \defi{cofinality} of an ordinal \alp\ is the smallest ordinal \del\ which is the order type of a cofinal subset of \alp. The cofinality of an ordinal is always a cardinal, but we will not need to make use of this fact.

For general graph-theoretic terminology we refer the reader to~\cite{diestelBook05}. Given a graph \g and a vertex $x\in V(G)$, we let $d_U(x)$ denote the number of neighbours of $x$ in a subset $U\subseteq V(G)$. For convenience, if $\cc$ is a subgraph of \G, or a set of subgraphs of \G, or a set of sets  of subgraphs of \G, we will simply write $d_\cc(x)$ ---to avoid more cumbersome notation of the form $d_{V(\bigcup \cc)} (x)$--- for the number of neighbours of $x$ (nested) in $\cc$.

A \defi{partition} of a set $V$ or a graph $G=(V,E)$ is a function $\pi: V \to \{0,1\}$. 
For a partition $\pi$ of $G$ and a vertex $x \in V(G)$, 
we say that a neighbour $y$ of $x$ is an \defi{\opp} (respectively, a \defi{friend}) of $x$ in $\pi$ if $\pi(y)\not=\pi(x)$ (resp.\ $\pi(y)=\pi(x)$). We write $a_\pi(x)$ for the number of \opps\ of $x$ in $\pi$. 
Moreover, if $U$ is a subset of $V(G)$ or a subgraph of $G$, then we write $a_\pi(x,U)$ for the number of \opps\ of $x$ in $\pi$ that lie in $U$. 
A vertex is \defi{happy} in a partition $\pi$ if 
it has at least as many opponents as it has friends. In particular, if a vertex $x$ has infinite degree then it 
is happy in $\pi$ if and only if $d(x)=a_\pi(x)$.
A partition is \defi{unfriendly} if every
vertex is happy in it, and it is \defi{unfriendly for} a vertex set $Y$ if every vertex in $Y$ is happy.

A graph \g is \defi{\prfr} if for every $U\subseteq V(G)$ and every partition $\pi$ of $U$ there is a partition $\pi'$ of \g extending $\pi$ so that every vertex in $V(G)\sm U$ is happy in $\pi'$. In particular, since we can choose $U$ to be the empty set, every \prfr\ graph has an unfriendly partition. Clearly, every finite graph is \prfr: given $U$ and $\pi$, choose $\pi'$ so as to maximize the number of edges incident with $V(G)\setminus U$ whose endvertices lie in different partition classes.

Schmidt~\cite{schmidt83} observed that it is possible to construct all rayless graphs by a recursive, transfinite procedure, starting with the class of finite graphs and then, in each step, glueing graphs constructed in previous steps along a common finite vertex set, to obtain new rayless graphs. We are going to generalise Schmidt's construction by replacing the class of finite graphs by larger classes. 

In order to do so, we call a class $\cu$ of graphs \defi{\klasse} if 
it is closed under taking finite disjoint unions and also
under adding finitely many new vertices and joining them arbitrarily to each other and to the old vertices. Starting with any \klasse\ class $\cu$ it is possible to recursively construct a larger class  $\overline{\cu}$ as follows.

\begin{definition}\label{def:clu}
Let $\cu$ be a \klasse\ class of graphs. 
For every ordinal $\mu$ we define recursively a class $\mathcal U(\mu)$:
\begin{itemize}
\item $\mathcal U(0):=\mathcal U$; and
\item if $\mathcal U(\lambda)$ is defined for every $\lambda<\mu$,
we include a graph $G$ in $\mathcal U(\mu)$ if it has a finite
vertex set $S$ such that for every component $C$ of $G-S$ there is an ordinal $\lambda<\mu$ such that $\mathcal U(\lambda)$ contains $C$.
\end{itemize}
The
\defi{closure $\overline{\mathcal U}$ of $\mathcal U$} is the union of the classes $\mathcal U(\lambda)$ for all ordinals \lam.
For a graph \g in $\overline{\mathcal U}$  we define its \defi{rank} to be the  smallest
ordinal $\rk_\mathcal U(G)=\rk(G)$ such that $\mathcal U(\rk(G))$ contains \G. 
\end{definition}

The prime example of a \klasse\ class is the class $\cf$ of finite graphs, and the rayless graphs are precisely those in~$\overline{\cf}$~\cite{schmidt83}. The locally finite graphs clearly do not form a \klasse\ class; but the graphs with only finitely many vertices of infinite degree do, and so do the graphs with only finitely many vertices of finite degree. Another class, which contains all these, will be defined in \Sr{secLarge}.

We note two properties of the above construction that are not hard to prove.

\begin{lemma}\label{S}
Let $G$ be a graph in $\overline{\mathcal U}$ for some \klasse\
class $\mathcal U$ of graphs, and let $S$ be a finite vertex
set so that $\rk(C)<\rk(G)$ for every component $C$ of $G-S$.
\begin{enumerate}[\rm (i)]
\item \label{ii} If $C_1,\ldots,C_n$ are components of $G-S$
then $\rk(G[S\cup\bigcup_{i=1}^n C_i])<\rk(G)$.
\item \label{i} If $S$ is  $\subseteq$-minimal with the property that $\rk(C)<\rk(G)$
for each component $C$ of $G-S$ then each vertex in $S$ has infinite degree.
\end{enumerate} 
\end{lemma} 
\begin{proof}
To see \eqref{ii}, we show that $\rk(G[S\cup\bigcup_{i=1}^n C_i])\le \max_i \rk(C_i)$ which is by assumption smaller than $r(G)$. Suppose first that all $C_i$ have rank $0$, i.e.\ lie in~$\cu$. Then so does $G[S\cup\bigcup_{i=1}^n C_i]$ as $\cu$ is \klasse. Otherwise, \fe\ $1\leq i \leq n$ with $r(C_i)>0$ let $S_i$ be a finite vertex set in $C_i$ \st\ the rank of every component of $C_i - S_i$ is less than $\rk(C_i)$. Then $S\cup \bigcup S_i$ is a finite set witnessing the fact that $\rk(G[S\cup\bigcup_{i=1}^n C_i])\leq \max_i \rk(C_i)$.

For \eqref{i}, note that if $x\in S$ has finite degree then the components of $G-(S\setminus\{x\})$ are precisely the components of $G-S$ that are not adjacent to $x$, plus one new component obtained by merging $x$ and the finitely many components of $G-S$ adjacent to $x$. By an argument similar to that of~\eqref{ii}, the rank of this new component is at most the maximum of the ranks of the components that got merged. But then $r(C)<r(G)$ for every component $C$ of $G-(S\setminus\{x\})$, contradicting the minimality of $S$ with this property.
\end{proof}

Here is the main result of this paper:

\begin{theorem}\label{main}
If $\cu$ is a \klasse\ class of graphs and every graph in $\cu$ is \prfr, then every graph in $\clu$ is \prfr.
\end{theorem} 

\noindent
As every finite graph is \prfr\ and every \prfr\ graph has an unfriendly partition, Theorem~\ref{main} with $\mathcal U = \mathcal F$ implies Theorem~\ref{rayless}.

\medskip

We prove \Tr{main} in \Sr{secMain}. In \Sr{secLarge} we use it to prove that a considerably larger class than the rayless graphs have unfriendly partitions.

\section{Proof of the main result} \label{secMain}

We now prove \Tr{main}. 

\begin{proof}[Proof of \Tr{main}]
Let $G=(V,E)$ be a graph in $\overline{\mathcal U}$.
We perform transfinite induction on $\rk(G)$. If $\rk(G)=0$ then \g is \prfr\ by  assumption, so suppose that $\rk(G)=\alpha>0$ and the assertion is true for all ordinals 
$\beta<\alpha$. Let $S\subseteq V(G)$ be a minimal separator such that $\rk(C)<\alpha$ for every component of $G - S$, and let $\ccc$ be the set of those components. We will prove that \g is \prfr\ by a second transfinite induction on $\kap:= |\ccc|$. By \Lr{S}, \kap\ must be infinite. We will treat the first step of the induction (i.e.\ the case $\kap=\oo$) more or less in the same way as the induction step; so assume from now on that either $\kap=\oo$ or $\kap>\oo$ and, by the induction hypothesis, any subgraph of \g spanned by $S$ and less than $\kap$ elements of $\ccc$ is \prfr; by \Lr{S}, we can assume this even if $\kap=\oo$.

Let $S_0$ be the set of vertices $s\in S$ for which there is a subset $\ccc_s$ of $\ccc$ of cardinality less than $\kap$ with $d_{ \ccc_s}(s)=d(s)$;
then let $S_1:=S\sm S_0$. Let $\cx_0:= \bigcup_{s\in S_0} \ccc_s$, and assume, without loss of generality, that 
\begin{equation} \label{snull} 
\begin{minipage}[c]{0.85\textwidth}
        $N(s)\subseteq S\cup \bigcup\mathcal X_0$ for every $s\in S_0$ with 
        $d(s)<\kappa$,
\end{minipage}\ignorespacesafterend  
\end{equation}          
where $N(s)$ denotes the set of vertices adjacent with $s$.

We claim that 
\labtequ{si}{if $s\in S_1$ then either $d(s)=\kap$ or $d(s)$ is a singular cardinal of cofinality $\kap$.}
Indeed, all vertices in $S$ with degree less than $\kap$ lie in $S_0$, thus $d(s)\geq \kap$.  If $s$ has degree $\lambda>\kap$ then $\lambda$ is a singular cardinal of cofinality at most $\kap$ as it is the sum of $\kap$ cardinals smaller than $\lambda$, namely of $d_C(s)$ for all $C\in\ccc$ (we know that $d_C(s)< \lam$ because $s\not\in S_0$). It remains to verify that the cofinality of $\lambda$ is not smaller than $\kap$. 

Suppose, on the contrary, that $\lambda$ is the sum of cardinals $\beta_\alp,~\alp<\gam,$ where $\gam<\kap$ and every $\beta_\alp$ is smaller than $\lambda$. Then \fe\ \alp\ \ti\ a component $C^\alp \in \cc$ \st\ $d_{C^\alp}(s)\geq \beta_\alp$; for if not, then $\lam \leq |\cc| \cdot \beta_\alp < \lam$, a contradiction. But then, the components $C^\alp, \alp<\gam$ contain $\lam= d(s)$ many of the neighbours of $s$, and as $\gam<\kap$ this contradicts $s\not\in S_0$.

It is possible to partition $\ccc\sm \cx_0$ into $\kap$ finite sets $X_\bet, \bet<\kap$ so that 
\begin{equation} \label{cof} 
\begin{minipage}[c]{0.85\textwidth}
        for every $s\in S_1$ and every  subset $\cz$ of $\ofml{X}{\kap}$ with $|\cz|=\kap$  there holds $d_{\cz}(s)=d(s)$.
\end{minipage}\ignorespacesafterend  
\end{equation}
Indeed, given $s\in S_1$, it is easy,  in both cases of~\eqref{si}, to find a sequence $\oseq{C^s}{\kap}$ of distinct elements of \ccc\ such that for every subset $\mathcal Z'$ of $\{C^s_\bet \mid \bet<\kap\}$ 
with $|\mathcal Z'|=\kap$ there holds $d_{ \mathcal Z'}(s)=d(s)$. 
Furthermore, let $\oseq{D}{\lam}$ be a sequence indexed by some ordinal $\lam\leq \kap$ containing all elements $D$ of \ccc\ that are not in a sequence $\oseq{C^s}{\kap}$ for any $s\in S_1$. 
Depending on whether $\kappa=\omega$ or not, we now define the 
$X_i$, $i<\kappa$, recursively as follows.
If $\kap>\oo$, then for every $s\in S_1$ we let $C^s_{m_i}$ be the first member of $\oseq{C^s}{\kap}$ that is not an element of $\bigcup_{\bet<i} X_\bet$, and let $X_i:=\sgl{D_i} \cup \{C^s_{m_i} \mid s\in S_1\}$ if $i<\lam$ and $X_i:=\{C^s_{m_i} \mid s\in S_1\}$ otherwise.
If $\kap=\oo$, we define $X_i$ in a slightly different way. Namely, we let $X_i:=\sgl{D_i} \cup \bigcup_{s\in S_1}\ccc^s_i$ if $i<\lam$ and $X_i:=\bigcup_{s\in S_1}\ccc^s_i$ otherwise, where 
$\ccc^s_i$ is a finite subset of $\ofml{C^s}{\kap}\sm\bigcup_{\bet<i} X_i$ containing $C^s_{m_i}$ such that $d_{\ccc^s_i}(s)>d_S(s)+d_{\mathcal X_0}(s)$; easily, such a $\ccc^s_i$ exists by~\eqref{si} and the definition of $S_1$. It is now easy to check that~\eqref{cof} is indeed satisfied, and moreover for the case $\kap=\oo$ we have in addition, \fe\  $\cz$ and every $s$ as in~\eqref{cof}, that
\begin{equation} \label{kapom} 
\begin{minipage}[c]{0.85\textwidth}
		$d_{ X}(s)>d_S(s)+ d_{\mathcal X_0}(s)$ for every $X\in\cz$.
\end{minipage}\ignorespacesafterend  
\end{equation}

Now let $U\subseteq V$, and let $\pi_U$ be any partition of $U$. We have to show that there is a partition $\chi$ of $V$ extending $\pi_U$ that is unfriendly for $V\sm U$. 

For every $\beta< \kap$, let $G_\beta$ be the subgraph of \g spanned by $S$ and all components in $\cx_0 \cup X_\beta$. 
Note that $\rk(G_\beta)\le \rk(G)$ for every $\beta$, and $G_\beta-S$ has less components than $G-S$ because $|\mathcal X_0|<\kappa$ and $|X_\beta|<\omega$. 
Thus by our second induction hypothesis ---or by our first induction hypothesis if $r(G_\beta) < r(G)$--- there is a partition $\rho_\beta$ of $G_\beta$ that extends $\pi_U \restr (V(G_\beta)\cap U)$ and is unfriendly for $V(G_\beta) \sm U$.

Define the \defi{$\cx_0$-stamp} of $\rho_\beta$ to be the set of vertices $s\in S_0$ that are given $d(s)$ opponents within $\bigcup\cx_0$ by $\rho_\beta$, i.e.\ those for which $a_{\rho_\bet}(s, \bigcup \cx_0)=d(s)$ holds. Since $S$ is finite, there is a partition $\rho_S$ of $S$, a subset $S'$ of $S_0$, and a subsequence $(\rho_\alpha)_{\alpha \in \cj}$ of $\oseq{\rho}{\kappa}$ with $|\cj|=\kap$ such that $\rho_\alpha \restr S = \rho_S$ and the $\cx_0$-stamp of $\rho_\alpha$ is $S'$ for every $\alpha \in \cj$. Using this subsequence, we will now construct a partition $\tau$ of \g that is unfriendly for all but finitely many of the vertices in $V\sm U$, and later modify $\tau$ to obtain the desired partition $\chi$ that is unfriendly for $V\sm U$. For every $\alpha \in \cj$, we partition $X_\alpha$ as in $\rho_\alpha$; formally, \fe\ vertex $x$ in $\bigcup X_\alpha$, let $\tau(x):=\rho_\alpha(x)$. We proceed similarly with $\cx_0$: we pick some 
$\alpha_0\in \cj$ and then \fe\ vertex $x$ in $\bigcup \cx_0$, we let $\tau(x):=\rho_{\alpha_0}(x)$. In order to partition the remaining vertices of the graph, we   define the partition $\pi_{U\cup S}:=\pi_U \cup \rho_S$; this is well-defined since $\rho_S$ is by construction compatible with $\pi_U$. Now for every $\beta \in \kap \sm \cj$ we pick a partition $\rho'_\beta$ of the graph $G'_\beta:=G[S\cup V(\bigcup X_\beta)]$ that extends $\pi_{U\cup S} \restr G'_\beta$ and is unfriendly for 
$V(G'_\beta)\sm (U\cup S)$ --- such a partition exists by our first inductive hypothesis, since $\rk(G'_\beta)<\rk(G)$ --- and again we partition the vertices in $\bigcup X_\beta$ according to $\rho'_\beta$: \fe\ vertex $x$ in $\bigcup X_\beta$, we let $\tau(x):=\rho_\beta(x)$. 

By construction, $\tau$ extends $\pi_U$. Moreover, it is straightforward to check that 
\begin{equation} \label{rho} 
\begin{minipage}[c]{0.85\textwidth}
        $\tau$ is unfriendly for $V \sm (S_1 \cup U)$.
\end{minipage}\ignorespacesafterend  
\end{equation}
Since the $\rho_\alpha$ have the same $\cx_0$-stamp \fe\ $\alpha \in \cj$, it follows that
\begin{equation} \label{sign} 
\begin{minipage}[c]{0.85\textwidth}
        \fe\ $s\in S_0$, either $a_{\tau}(s,\bigcup \cx_0)=d(s)$ or \fe\ $\alpha \in \cj$ there holds $a_{\tau}(s,\bigcup X_\alpha)=d(s)$ (or both).
\end{minipage}\ignorespacesafterend  
\end{equation}  

Let $F\subseteq S_1$ be the set of vertices of $V\setminus U$ that are not happy
in $\tau$. We claim that 
\begin{equation} \label{kappa} 
\begin{minipage}[c]{0.85\textwidth}
        \fe\ $r\in F$ there holds $d(r)=\kap$.
\end{minipage}\ignorespacesafterend  
\end{equation}
Indeed, pick any $r\in S_1\setminus U$ with $d(r)>\kappa$. We need to show that
this choice implies $r\notin F$. 
By the definition of $S_0$ and since $r\notin S_0$, $r$ cannot have full 
degree in $\bigcup\mathcal X_0$, i.e.\ $\delta:=d_{\mathcal X_0}(d)<d(r)$.
Define $\mathcal Z'$ to be the set of those $X_\alpha$, $\alpha\in\mathcal J$,
for which $d_{ X_\alpha}(r)\leq\delta$. 
Since $|\mathcal Z'|\leq |\mathcal J|=\kappa<d(r)$ and $\delta<d(r)$,
we obtain that
$d_{\mathcal Z'}(r)\leq \kappa\cdot\delta<d(r)$.
Hence, from \eqref{cof} it follows that $|\mathcal Z'|<\kappa$
and thus the 
set $\mathcal Z:=\{X_\alpha:\alpha\in\mathcal J\}\setminus\mathcal Z'$
has cardinality $\kappa$. Using~\eqref{cof} again 
yields $d_{\mathcal Z}(r)=d(r)$.

Now consider one such $X_\alpha\in\mathcal Z$.
Recall that $r$ is happy in the partition $\rho_\alpha$ of
the graph $G_\alpha$, which we defined to be the induced subgraph 
on $S$ together with all components in $\mathcal X_0\cup X_\alpha$.
Since $d_{ X_\alpha}(r)>\delta=d_{\mathcal X_0}(r)$,
this means that $r$ has, with respect to $\rho_\alpha$,
$d_{ X_\alpha}(r)$ opponents in $\bigcup X_\alpha$.
Since $\tau$ and $\rho_\alpha$ coincide on $\bigcup X_\alpha$
by the definition of $\tau$, this implies
that $a_\tau(r,\bigcup X_\alpha)=d_{ X_\alpha}(r)$.
As, on the other hand, we have found that 
$d_{\mathcal Z}(r)=d(r)$, we deduce that $r$ has $d(r)$
opponents in $\tau$ and is therefore happy.
Since $F$ comprises unhappy vertices, we have proved~\eqref{kappa}.

We claim that $F$ is empty if $\kap=\oo$. Indeed, if $\kap=\oo$ then every vertex $r$ in $F$ has degree $\oo$ by~\eqref{kappa} and so $d_{\cx_0}(r)$ is finite. Thus by~\eqref{kapom} and the fact that $\tau$ coincides with the 
partition $\rho_\alpha$ of $G_\alp$, which is unfriendly for $V(G_\alpha)\sm U$, 
$\tau$ gives $r$ an opponent in $X_\alp$ for every $\alp\in \cj$. But then $r$ has $|\cj|=\kap=\oo$ opponents in $\tau$, contradicting the assumption that $r\in F$. Thus $F=\emptyset$ if $\kap=\oo$, which means that $\tau$  
is unfriendly for $V\setminus U$ and we are done.
So we may from now on assume that 
\begin{equation} \label{uncntbl} 
\begin{minipage}[c]{0.85\textwidth}
$\kappa$ is an uncountable cardinal.
\end{minipage}\ignorespacesafterend  
\end{equation}

Denote by $\tau'$ the partition obtained from $\tau$ by \defi{flipping $F$},
i.e.\ by moving each vertex in $F$ from its partition class in $\tau$ 
to the other partition class (formally, by changing the image under $\tau$ of every vertex in $F$).
Then, $\tau'$ is unfriendly for $S\sm U$, as all vertices in $S$ have infinite degree by \Lr{S}, but there might be vertices in the rest of the graph that were made unhappy by flipping $F$. 
It follows from~\eqref{rho} and~\eqref{kappa} that the set $\cy$ of elements of 
$\ccc\sm \cx_0$ (we will turn our attention to $\cx_0$ later) that  contain vertices that are unhappy in $\tau'$ has cardinality less than \kap, because each vertex in $F$ had less than \kap\ \opps\ in $\tau$. 

We now modify $\tau'$ within $\cy$ to obtain a new partition $\phi$ of \g that is unfriendly for $V \sm (V(\bigcup\cx_0)\cup U)$. 
For every component $C\in \cy$, let $U':=(U\cap V(C)) \cup S$,  and
pick a partition $\phi_C$ of $G[S \cup V(C)]$ that extends $\tau' \restr U'$ 
and is unfriendly for $V(C) \sm U'$; such a partition exists by our first induction hypothesis as $\rk(C)<\rk(G)$. 
Then, \fe\ $x\in V(C)$ put $\phi(x):= \phi_C(x)$. Vertices that are not contained within $\cy$ we leave unchanged: \fe\ $x\in V \sm V(\bigcup \cy)$ put $\phi(x):= \tau'(x)$.

We claim that 
\begin{equation} \label{phi} 
\begin{minipage}[c]{0.85\textwidth}
$\phi$ is unfriendly for $V\sm (V(\bigcup\cx_0) \cup U)$. 
\end{minipage}\ignorespacesafterend  
\end{equation}
Indeed, $\phi$ is by construction unfriendly for $V\sm (V(\bigcup\cx_0) \cup U \cup S)$. To see that it is unfriendly for $S\sm U$ as well, recall that $\tau'$ was unfriendly for $S\sm U$, and suppose there is an $s\in S$ unhappy in $\phi$. Since $s$ was happy in $\tau'$, it follows that $s$ has $d(s)$ many neighbours within $\cy$, which puts $s$ in $S_0$ as $|\mathcal Y|<\kappa$. 
But then by~\eqref{sign} if $s$ is not given $d(s)$ \opps\ in $\cx_0$ by $\tau$, and thus also by $\phi$, then $s$ is given $d(s)$ \opps\ in each of the \kap\ many $X_\alp$ with $\alpha\in\cj$ that do not contain a component in $\cy$, a contradiction. 
This establishes Claim~\eqref{phi}.

Finally, we turn our attention to the unhappy vertices in $\bigcup\mathcal X_0$.
These vertices were happy in $\tau$ and thus had an opponent in $\tau$
that is lost in $\phi$, which means that they are adjacent
to a vertex in $F$. Observe that
\begin{equation} \label{gamma} 
\begin{minipage}[c]{0.85\textwidth}
$\gamma:=\max\{\omega,\max_{r\in F} a_\tau(r,V(\bigcup\mathcal X_0))\}<\kappa.$
\end{minipage}\ignorespacesafterend  
\end{equation}
Indeed, recall that $\omega<\kappa$ by~\eqref{uncntbl}. Moreover,
by the definition of $S_1$ each vertex $r\in F\subseteq S_1$ has less than 
full degree in $V(\bigcup\mathcal X_0)$, and thus
$a_\tau(r,V(\bigcup\mathcal X_0))\leq d_{\mathcal X_0}(r)<d_G(r)$.
Since $d_G(r)=\kappa$ by~\eqref{kappa} the claim follows.

Let $H:=G[V(\bigcup \cx_0) \cup S]$ and define $S'_0$ to be the 
set of vertices $s\in S_0$ with $a_\phi(s, V(\bigcup \cx_0))<d_G(s)$.
Putting $\tilde U:=(U\cap V(H))\cup S_0'\cup S_1$ we shall find 
a partition $\rho$ of $H$ with the following properties.
\begin{enumerate}[(i)]
\item\label{ai} $\rho$ is unfriendly for $V(H)\setminus\tilde U$;
\item\label{aii} $\rho$ extends $\phi\upharpoonright\tilde U$; and 
\item\label{aiii} for every $x\in V(H)$ if $d_H(x)>\gamma$ then $\rho(x)=\phi(x)$.
\end{enumerate}
Intuitively, what $\rho$ accomplishes is to repartition the vertices in $\cx_0$ and part of the vertices in $S_0$ in such a way that firstly, every vertex in $\cx_0$ or $S$ is made happy, and secondly, any vertex in $S_0$ whose partition class we changed has relatively small degree. The latter condition on the degree will then help make sure that the happiness of the vertices in $\bigcup X_\alp$ is not affected by the repartitioning imposed by $\rho$.

Before we construct $\rho$ let us check that it would indeed allow us to finish the proof.
For this, we modify $\phi$ within $H$ according to $\rho$ to obtain a new partition 
$\chi$ of \G: \fe\ $x\in V(H)$ let $\chi(x):=\rho(x)$ and \fe\ $x \in V\sm V(H)$ 
let $\chi(x):=\phi(x)$. The following table summarises the definitions of some of the partitions of \g we have defined so far, and indicates which vertices are happy in each partition.

\bigbreak

\newcommand{\tik}{\checkmark}
\newcommand{\kit}{\ensuremath{\times}}
\begin{tabular}{c|c|c|c|c|c|l}
	&  $\cx_0$	&$\cc\sm \cx_0$	&	$S_0$	&	$F$	&	$S_1 \sm F$	&	\\
	\hline
$\tau$&	\tik	&	\tik		&	\tik	&	\kit&	\tik		& composed from the $\rho_\bet$\\
	\hline
$\tau'$& ?		&	?			&	\tik	&	\tik&	\tik		& obtained from $\tau$ by flipping $F$ \\
	\hline
$\phi$ & ?		&	\tik		&	\tik	&	\tik&	\tik		& obtained from $\tau'$ by changes in $\cy$ \\
	\hline
$\chi$ &	\tik&	\tik		&	\tik	&	\tik&	\tik		& obtained from $\phi$ by changes in $H$
\end{tabular}

\bigbreak

We claim that $\chi$ is unfriendly for $V\sm U$. 
Indeed, any vertex $x\in V(\bigcup\mathcal X_0)\sm U$ is clearly happy by
condition~\eqref{ai}. For a vertex $s\in S\sm U$ we distinguish three cases.
If $s\in S_1$ then $s$ has less that $d_G(s)$ neighbours in 
$\bigcup\cx_0$ by the definition of $S_1$, and thus by~\eqref{aii} 
$s$ is happy in $\chi$ as it was happy in $\phi$. If $s\in S'_0$
then, by the definition of $S'_0$ and~\eqref{phi}, it had $d_G(s)$ opponents outside
$\bigcup\mathcal X_0$ in $\phi$. As $\phi$ and $\chi$ coincide outside
$\bigcup\mathcal X_0$ and as $\phi(s)=\chi(s)$ by~\eqref{aii}, the
vertex $s$ is happy in $\chi$. Finally, if $s\in S_0\sm S'_0$ then
$\chi(s)=\rho(s)$. As $s$ is happy in $\rho$ (within $H$) by~\eqref{ai}
and as $\chi$ coincides with $\rho$ on $H$, we find that $s$ is happy in $\chi$
as well (note that $d_H(s)=d_G(s)$ since $s\in S_0$).
It remains to check that every vertex in $V\sm (S\cup V(\bigcup\mathcal X_0)\cup U)$
is happy too. So consider a vertex $x$ in $\bigcup X_\alp \sm U$ for  some
ordinal $\alpha$. 
Such an $x$ was happy in $\phi$, but since we changed the partition of $S$ we may have made it unhappy in $\chi$. This is however not the case: if $s\in S$ and $\phi(s)\neq \chi(s)$, then by~\eqref{aii} and~\eqref{aiii}  we have $s\in S_0$ and 
$d_G(s)=d_H(s)\leq \gamma < \kap$, where  the last inequality is~\eqref{gamma}. 
Thus by~\eqref{snull} $s$ has all its neighbours in $\bigcup\cx_0$, 
which means that  $s$ is not a neighbour of $x$. 
In conclusion, $x$ is still happy in $\chi$ and this completes the proof that $\chi$ is unfriendly for $V\sm U$. 

\medskip

To finish the proof of the theorem we still need to construct the partition~$\rho$.
Define $A$ to be the set of all vertices $a$ in $H\sm \tilde U$ with $d_H(a)\leq\gamma$. 
Denote by $K$ the union 
of the vertex sets of 
all the components of $H[A]$ that 
contain an unhappy vertex (with respect to $\phi$). Observe that such 
an unhappy vertex must have been an opponent in $\tau$ of some vertex in $F$. Thus by~\eqref{gamma} there are at most \gam\ components of this kind, and as each of them has at most \gam\ vertices we obtain 
\labtequ{kaga}{$|K| \leq \gam$.}

By our second induction hypothesis the graph $H$ is \prfr, thus there is a partition $\rho$ of $H$ extending $\phi\upharpoonright \tilde U\cup(V(H)\sm K)$
such that every vertex in $K\sm\tilde U$ is happy.
Clearly, $\rho$ satisfies~\eqref{aii} and~\eqref{aiii}.

To see that every vertex $x$ in $V(H) \sm (K \cup \tilde U)$ is happy too, 
which would complete the proof that \eqref{ai} is satisfied as well, we note that  $\rho$ differs from $\phi$ only within $K$, and distinguish two cases. 
If $d_H(x)\leq \gamma$, then since $x\not\in K$, its neighbourhood $N_H(x)$ does not meet $K$ by the definition of $A$, thus $\phi$ and $\rho$ coincide in $N_H(x)$. By the definition of $K$ it follows that $x$ was happy in $\phi$, thus it is also happy in $\rho$ as its neighbourhood was not repartitioned. 
If, on the other hand, $d_H(x)> \gamma$, then $x$ was happy in $\phi$ because $|F|$ is finite, \gam\ is infinite, and $x$ was happy in $\tau$. But then $x$ is happy in $\rho$ by~\eqref{kaga} since only vertices in $K$ were repartitioned while changing $\phi$ to $\rho$. 
\end{proof}

\section{More graphs with \ufp s} \label{secLarge}

In the rest of the paper we shall strengthen Theorem~\ref{rayless} by applying Theorem~\ref{main} to a larger \klasse\ class than just~$\mathcal F$. A~proof of Theorem~\ref{brush} will follow as a corollary.

Given a graph $G=(V,E)$, 
let $V_\infty$ be the set of all vertices of infinite degree, and denote
by $V^*$ the set of those vertices in $V_\infty$ that have
only finitely many neighbours in $V_\infty$. 
Let $\cw$ be the class of countable graphs $G$ such that $V^*(G)$ is finite. 
In particular, $\cw$ contains all countable graphs \g for which $\vinf(G)$ or $V(G)\sm\vinf(G)$ is finite. Note that $\clx{\cw}$ also contains all uncountable graphs \g for which $\vinf(G)$ is finite, since deleting $\vinf(G)$ shows that $G$ has rank~1 in~$\clx{\cw}$.
Clearly, $\cw$ is \klasse. We claim that 

\begin{theorem}\label{W}
Every graph in $\overline\cw$ is \prfr. In particular, every such graph has an unfriendly partition.
\end{theorem} 

In order to prove this we will need the following lemma from \cite[Lemma~3]{AhMiUnf}.

\begin{lemma}\label{AMP}
Let $G=(V,E)$ be a countable graph and let $U$ be a subset of $V$ such 
that only finitely many vertices in $V\sm U$ have infinite degree (in $G$). 
Then for every partition $\pi$ of $G[U]$  there 
exists a partition $\pi'$ of $G$ that extends $\pi$ and
which is unfriendly for $V\sm U$.
\end{lemma} 
\note{PROOF: A partition $\pi$ of $G$ is \defi{strongly maximal} if for every finite $W\subset V(G)$ and every $W'\subseteq W$, flipping $W'$ does not increase the number of cross-edges incident with $W$, which we call the \defi{happiness} of $W$ (in $\pi$) and denote by $h_{\pi}(W)$. Given a partition $\pi$ of a set $U\subset V$, begin by putting the vertices in $V_\infty\sm U$ all to side $0$ and then use compactness to extend this partition of $U\cup V_\infty$ to a \smp\ $\pi'$ on the rest of $V$. All vertices in $V\sm(U\cup V_\infty)$ are happy in $\pi'$, the only unhappy vertices in $V\sm U$ can lie in $V_\infty$. Let $F$ be the set of those vertices and consider the partition $\pi'*F$, that is $\pi'$ with $F$ flipped. Since the vertices in $F$ were unhappy in $\pi'$ there were only finitely many crossing edges incident with $F$, say $k$. Now flipping $F$ has made all vertices in $F$ happy, but may have made some vertices in $V\sm(U\cup V_\infty)$ unhappy. If there are no such unhappy vertices we have found a partition that is unfriendly on $V\sm U$, so we may assume that there is an unhappy vertex $v_1$. Flip $v_1$ to obtain a partition $\pi_1$ and let $v_2$ be a vertex in $V\sm(U\cup V_\infty)$ that is unhappy in $\pi_1$---if there is no such vertex, then $\pi_1$ is unfriendly on $V\sm U$. Assume that we can repeat this procedure $2k+1$ times, having flipped vertices $v_1,\dotsc,v_{2k+1}$ (not necessarily distinct---it is allowed to flip vertices back again) obtaining partitions $\pi_1,\dotsc,\pi_{2k+1}$. Let $W:=\{v_1,\dotsc,v_{2k+1}\}$ and let $W'\subset W$ be the set of vertices that appear in $v_1,\dotsc,v_{2k+1}$ an odd number of times. Then $\pi_{2k+1}$ is $(\pi'*F)*W'$, and the happiness of $W$ has increased by at least $2k+1$ (i.e.\ $h_{\pi_{2k+1}}(W)\ge h_{\pi*F}(W)+2k+1$), as $h_{\pi'*F}(W) < h_{\pi_1}(W) < h_{\pi_2}(W) < \dotsb < h_{\pi_{2k+1}}(W)$. But then flipping $W'$ in $\pi'$ increases the happiness of $W$: Every crossing edge incident with $W$ gained by flipping $W'$ in $\pi'*F$ that is not gained by flipping $W'$ in $\pi'$ clearly is an $F$--$W'$~edge as it would have been gained in both situations if it had no endvx in $F$ and it would not have been gained at all if it had no endvx in $W'$. Moreover, that such an edge is gained by flipping $W'$ in $\pi'*F$ means that its endvx in $W'$ lies in side $1$ in $\pi'*F$ (hence also in $\pi'$), as all of $F$ lies in side $1$ in $\pi'*F$. But then this edge was a crossing edge in $\pi'$ incident with $F$. Analogously, every crossing edge that is lost by flipping $W'$ in $\pi'$ either is also lost by flipping $W'$ in $\pi'*F$ or it is a crossing edge in $\pi'$ incident with $F$. As there are only $k$ crossing edges in $\pi'$ incident with $F$, by flipping $W'$ in $\pi'$ we gained all crossing edges we gained by flipping $W'$ in $\pi'*F$ minus at most $k$ edges, and we lost at most $k$ crossing edges more than by flipping $W'$ in $\pi'*F$, so $h_{\pi'*W'}(W)-h_{\pi'}(W)\ge h_{\pi'*F*W'=\pi_{2k+1}}(W)-h_{\pi'*F}(W)-2k\ge 1$, which contradicts the fact that $\pi'$ is a \smp\ on $V\sm(U\cup V_\infty)$. Hence the construction has to terminate after $m<2k+1$ steps, meaning that we have found a partition $\pi_m$ that is unfriendly for $V\sm U$.}

\begin{proof}[Proof of \Tr{W}]
Consider a graph $G$ in $\mathcal W$.
We will prove that \g has an \ufp. It is easy to modify this proof in order to show that \g is \prfr. The theorem then follows from Theorem~\ref{main}.

First, let us construct an unfriendly partition $\pi$ of $G[V_\infty\sm V^*]$.
Observe that $|V^*|<\infty$ implies that every vertex in $G[V_\infty\sm V^*]$
has infinite degree. Pick a sequence $(v_i)_{i\in\mathbb N}$ 
in which every vertex in $V_\infty\sm V^*$ appears infinitely often. 
We go through this sequence, and if $\pi(v_i)$ has not been defined yet
we set $\pi(v_i)=0$. Otherwise, we choose a neighbour $y\in V_\infty\sm V^*$
of $v_i$ with $\pi(y)$ still undefined and set $\pi(y):=1-\pi(v_i)$. It is now easy to check that $\pi$ is indeed unfriendly.

Next, Lemma~\ref{AMP} yields a partition $\pi'$ of $G$ extending $\pi$
so that $\pi$ is unfriendly for $V(G)\sm(V_\infty\sm V^*)$. Since 
every vertex in $V_\infty\sm V^*$ already had infinitely many 
opponents in $\pi$, it follows that $\pi'$ is unfriendly for all of $V(G)$.
\end{proof}

The countable graphs in $\overline{\mathcal W}$ can be characterised as follows.

\begin{proposition}\label{ccw}
A countable graph \g lies in $\clx{\cw}$ \iff\ it contains no comb with all teeth in $V^*(G)$.
\end{proposition} 
\begin{proof} 
Suppose there are \comment{DON'T NEED TO SAY "COUNTABLE" HERE}%
graphs in $\clx{\cw}$ that contain a comb  with all teeth in $V^*(G)$. Then, there is such a graph $G$ that has minimal rank among all those graphs; let $C$ be a comb in $G$ with all its teeth in $V^*(G)$. Clearly, \g does not have rank $0$, so there is a finite set $S$ of vertices such that all components of $G-S$ have rank smaller than $\rk(G)$. But one of the components contains a tail of $C$, contradicting the minimality of $\rk(G)$.

Conversely, let $G$ be a countable graph not in $\clx{\cw}$. 
Then
there is a component $C_0$ of $G$ that
does not lie in $\overline{\mathcal W}$ (otherwise $S=\emptyset$ 
is a separator as in Definition~\ref{def:clu} showing that $G\in\overline{\mathcal W}$).
Since $C_0\notin\mathcal W$ although $C_0$ is countable, \ti\ a vertex $v_0\in V^*$ in $C_0$. 
At least one component $C_1$ of $C_0-v_0$ is not in 
$\overline{\mathcal W}$, and thus contains a vertex $v_1\in V^*(G)$.
Let $P_1$ be a $v_0$--$v_1$~path in $C_1\cup\{v_0\}$.
Now recursively for $i=1,2,\dotsc$, let $C_i$ be a component of 
$C_{i-1}-P_{i-1}$ that is not in $\overline{\mathcal W}$,
let $v_{i}$ be a vertex of $V^*$ in $C_i$ and let $P_i$ be a 
$V(P_{i-1}\cap C_{i-1})$--$v_{i}$~path in $G[C_i\cup P_{i-1}]$. 
It is not hard to see that $\bigcup_{i<\oo}P_i$ is a comb with teeth $v_0,v_1,\dotsc$ in $V^*$.
\end{proof}

As remarked earlier, the class $\mathcal U$ of (countable or uncountable) graphs that have only finitely many vertices of infinite degree is contained in~$\clx{\cw}$. And as in the proof of \Prr{ccw} we see that any graph not in $\clx{\mathcal U}$ contains a comb whose leaves all have infinite degree; let us call such a comb a \defi{brush}. Graphs not containing a brush thus lie in~$\clx{\mathcal U}\subseteq \clx{\clx{\cw}}=\clx{\cw}$, and we obtain Theorem~\ref{brush} as a corollary of \Tr{W}:

\begin{corollary}\label{teeth}
Every graph not containing a brush has an \ufp.
\end{corollary}

\acknowledgement{We would like to thank Matt DeVos for inspiring discussions on this subject.}

\bibliographystyle{plain}
\bibliography{collective}
\end{document}